\documentclass[12pt,leqno]{article}
\usepackage[centertags]{amsmath}
\usepackage{amsfonts}
\usepackage{amssymb}
\usepackage{amsthm}
\renewcommand\baselinestretch{1.5}
\theoremstyle{plain}
\newtheorem{thm}{Theorem}[section]
\newtheorem{cor}[thm]{Corollary}
\newtheorem{lem}[thm]{Lemma}

\newtheorem{defn}[thm]{Definition}
\newtheorem{rem}[thm]{Remark}
\newtheorem{Exa}[thm]{Example}
\numberwithin{equation}{section}
\renewcommand\baselinestretch{1.5}
\begin{document}
\vspace{4cm} \vspace{4cm} \centerline{\bf On weakly $(m,n)-$closed $\delta-$primary ideals of commutative rings}\vspace{2cm}
 \centerline {\baselineskip.8cm
\centerline {  {Mohammad Hamoda$^1$ and Mohammed Issoual$^2$  } }}
\centerline {\baselineskip.8cm {$^1$Department of Mathematics, Faculty of Applied Science, Al-Aqsa University, Gaza, Palestine.}}
\centerline {\baselineskip.8cm{P.O. Box 4051, Gaza, Palestine }}
\centerline {\baselineskip.8cm{e-mail: ma.hmodeh@alaqsa.edu.ps }}
\centerline {\baselineskip.8cm{ORCID iD: https://orcid.org/0000-0002-5452-9220 }}
\baselineskip.8cm {$^2$Department of Mathematics, Crmef Rabat-Sale-Kenitra, Annex Khmisset, Rabat, Morocco.}\\
\centerline {\baselineskip.8cm{Lycee Abdellah Gennon, R404, Khemisset, Morocco }}
\centerline {\baselineskip.8cm{e-mail: issoual2@yahoo.fr }}
\centerline {\baselineskip.8cm{ORCID iD: https://orcid.org/0000-0002-9872-8221}}

\thispagestyle{empty}
\renewcommand\baselinestretch{1.5}
\vspace{2cm}
\begin{abstract}
Let $R$ be a commutative ring with $1\neq0$.  In this article, we introduce the concept of weakly $(m,n)-$closed $\delta-$primary ideals of $R$ and explore its basic properties.  We show that $I\bowtie^{f}J$ is a weakly $(m,n)-$closed $\delta_{\bowtie^{f}}-$primary ideal of $A\bowtie^{f}J$ that is not $(m,n)-$closed $\delta_{\bowtie^{f}}-$primary if and only if $I$ is a weakly $(m,n)-$closed $\delta-$primary ideal of $A$ that is not $(m,n)-$closed $\delta-$primary and for every $\delta$-$(m,n)$-unbreakable-zero element $a$ of $I$ we have $(f(a)+j)^{m}=0$ for every $j\in~J$, where $f:A\rightarrow B$ is a homomorphism of rings and $J$ is an ideal of $B.$  Furthermore, we provide examples to demonstrate the validity and applicability of our results.
\end{abstract}
\vspace{1cm} \noindent {\bf "2020 Mathematics Subject Classification"}: 13C05; 13C13; 13F05.
\\ \noindent{\bf Keywords:}
$\delta-$primary ideal, weakly $2-$absorbing ideal, weakly $n-$absorbing ideal.\\
\section{Introduction}
Throughout this article all rings are commutative with nonzero identity, all modules are unitary and all ring homomorphisms preserve the identity.  Recently, ring theorists have been interested in the class of prime ideals and modules and its generalizations.  The notion of weakly prime ideals was introduced by Anderson and Smith in \cite{p} and further studied by several authors, (see for example \cite{q,i}).  Badawi in \cite{c} generalized the notion of prime ideals in a different way, called $2-$absorbing ideal.  Many authors studied on this issue, (see for example \cite{b,e,h,k,w,n}).  The notion of $\delta-$primary ideals in commutative rings was introduced by Zhao in \cite{o}.  This concept is considered to unify prime and primary ideals.  Many results of prime and primary ideals are extended to these structures, (see for example \cite{j}, \cite{m}).  Our aim of this article is to introduce a new class of ideals that is closely to the class of weakly $n-$absorbing ideals, called weakly $(m,n)-$closed $\delta-$primary ideals.  The motivation of this article is to complete what has been studied in \cite{t} and to develop related results.  The remains of this article is organized as follows.  In section $2$, we give some basic concepts and results that are indispensable in the sequel of this article.  Section $3$, is devoted to the main results concerning weakly $(m,n)-$closed $\delta-$ primary ideals.  Section 4 concerns weakly $(m,n)-$closed $\delta-$primary ideals in amalgamated algebra.  Section $5$, concerns the conclusion.
\section{Preliminaries}
In this section, we state some basic concepts and results related to weakly $(m,n)-$closed $\delta-$primary ideals.  We hope that this will improve the readability and understanding of this article.  Let $R$ be a commutative ring and $I$ be an ideal of $R$.  An ideal $I$ is called proper if $I\neq~R$.  Let $I$ be a proper ideal of $R$.  Then, the radical of $I$ is defined by $\{x\in~R~\mid~\exists~n\in\mathbb{N},~x^n\in~I\}$, denoted by $\sqrt{I}$ (note that $\sqrt{R}=R$ and $\sqrt{0}$ is the ideal of all nilpotent elements of $R$).  For the ring $R$, we shall use $Nil(R)$, $J(R)$ and $char(R)$ to denote the set of all nilpotent elements of $R$, the set of all maximal ideals of $R$ and the characteristic of $R$, respectively.
\begin{defn}(\cite{c}, \cite{g})
A proper ideal $I$ of $R$ is called a $2-$absorbing ideal (respectively, $2-$absorbing primary ideal) of $R$ if whenever $abc\in~I$ (respectively, $0\neq~abc\in~I$) for some $a,b,c\in~R$, implies $ab\in~I$ or $ac\in~I$ or $bc\in~I$ (respectively, $ab\in~I$ or $ac\in~\sqrt{I}$ or $bc\in~\sqrt{I}$).
\end{defn}
\begin{defn}\cite{b}
Let $n$ be a positive integer.  A proper ideal $I$ of $R$ is called an $n-$absorbing ideal (respectively, strongly $n-$absorbing ideal) of $R$ if whenever $x_1...x_{n+1}\in~I$ for some $x_1,...,x_{n+1}\in~R$ (respectively, $I_1...I_{n+1}\subseteq~I$ for some ideals $I_1,...,I_{n+1}$ of $R$), then there are $n$ of the $x_i~'s$ (respectively, $n$ of the $I_i~'s$) whose product is in $I$.
\end{defn}
Thus, a $1-$absorbing ideal is just a prime ideal.
\begin{defn}\cite{w}
Let $n$ be a positive integer.  A proper ideal $I$ of $R$ is called a weakly $n-$absorbing ideal (respectively, strongly weakly $n-$absorbing ideal) of $R$ if whenever $0\neq~x_1...x_{n+1}\in~I$ for some $x_1,...,x_{n+1}\in~R$ (respectively, $0\neq~I_1...I_{n+1}\subseteq~I$ for some ideals $I_1,...,I_{n+1}$ of $R$), then there are $n$ of the $x_i~'s$ (respectively, $n$ of the $I_i~'s$) whose product is in $I$.
\end{defn}
Thus, a weakly $1-$absorbing ideal is just a weakly prime ideal.
\begin{defn}\cite{a}
Let $m$ and $n$ be positive integers.  A proper ideal $I$ of $R$ is called a semi$-n-$absorbing ideal of $R$ if whenever $x^{n+1}\in~I$ for some $x\in~R$ implies $x^n\in~I$.  More generally, a proper ideal $I$ of $R$ is called an $(m,n)-$closed ideal of $R$ if whenever $x^m\in~I$ for some $x\in~I$ implies $x^n\in~I$.
\end{defn}
\begin{defn}\cite{r}
Let $m$ and $n$ be positive integers.  A proper ideal $I$ of $R$ is called a weakly semi$-n-$absorbing ideal of $R$ if whenever $0\neq~x^{n+1}\in~I$ for some $x\in~R$ implies $x^n\in~I$.  More generally, a proper ideal $I$ of $R$ is called a weakly $(m,n)-$closed ideal of $R$ if whenever $0\neq~x^m\in~I$ for some $x\in~I$ implies $x^n\in~I$.
\end{defn}
\begin{defn}\cite{o}
Let $Id(R)$ be the set of all ideals of $R$.  A function $\delta:Id(R)\longrightarrow~Id(R)$ is called an expansion function of $Id(R)$ if it has the following two properties: $I\subseteq\delta(I)$ and if $I\subseteq~J$ for some ideals, $I$, $J$ of $R$, then $\delta(I)\subseteq\delta(J)$.  A proper ideal $I$ of $R$ is called a $\delta-$primary ideal of $R$ if whenever $xy\in~I$ for some $x,y\in~R$ implies $x\in~I$ or $y\in\delta(I)$.
\end{defn}
\begin{defn}\cite{j}
A proper ideal $I$ of $R$ is called a $2-$absorbing $\delta-$primary ideal of $R$ if whenever $xyz\in~I$ for some $x,y,z\in~R$ implies $xy\in~I$ or $yz\in\delta(I)$ or $xz\in\delta(I)$.  A proper ideal $I$ of $R$ is called a strongly $2-$absorbing $\delta-$primary ideal of $R$ if whenever $I_1,I_2,I_3$ are ideals of $R$, $I_1I_2I_3\subseteq~I$, $I_1I_3\not\subseteq~I$ and $I_2I_3\not\subseteq~\delta(I)$, then $I_1I_2\subseteq\delta(I)$.
\end{defn}
The notions of $n-$absorbing $\delta-$primary ideals and weakly $n-$absorbing $\delta-$primary ideals are generalizations of the notions of $n-$absorbing primary ideals and weakly $n-$absorbing primary ideals respectively.  Recall the following definition.
\begin{defn}\cite{m}
A proper ideal $I$ of $R$ is called an $n-$absorbing $\delta-$primary ideal (respectively, weakly $n-$absorbing $\delta-$primary ideal) of $R$ if whenever $x_1...x_{n+1}\in~I$ (respectively, $0\neq~x_1...x_{n+1}\in~I$) for some $x_1,...,x_{n+1}\in~R$ implies $x_1...x_n\in~I$ or there exists $1\leq~k<n$ such that $x_1...\hat{x}_k...x_{n+1}\in\delta(I)$, where $x_1...\hat{x}_k...x_{n+1}$ denotes the product of $x_1...x_{k-1}x_{k+1}...x_{n+1}$.
\end{defn}
The notion of $(m,n)-$closed $\delta-$primary ideals is a generalization of the notion of $n-$absorbing $\delta-$primary ideals.  Recall the following definition.
\begin{defn}\cite{t}
Let $R$ be a commutative ring, $\delta$ be an expansion function of $Id(R)$, and $m$ and $n$ be positive integers. A proper ideal$I$ be of $R$ is called a semi$-n-$absorbing $\delta-$primary ideal of $R$ if whenever $a^{n+1}\in~I$ for some $a\in~R$, then $a^n\in\delta(I)$.  More generally, a proper ideal $I$ of $R$ is called an $(m,n)-$closed $\delta-$primary ideal of $R$ if whenever $a^m\in~I$ for some $a\in~R$, then $a^n\in\delta(I)$.
\end{defn}
\section{Properties of weakly $(m,n)-$closed $\delta-$primary ideals}
We start by the following definitions.
\begin{defn}
Let $R$ be a commutative ring, $I$ be a proper ideal of $R$, $\delta$ be an expansion function of $Id(R)$, and $m$ and $n$ be positive integers.\\
(1) $I$ is called a weakly semi$-n-$absorbing $\delta-$primary ideal of $R$ if whenever $0\neq~a^{n+1}\in~I$ for some $a\in~R$, then $a^n\in\delta(I)$.\\
(2) $I$ is called a weakly $(m,n)-$closed $\delta-$primary ideal of $R$ if whenever $0\neq~a^m\in~I$ for some $a\in~R$, then $a^n\in\delta(I)$.
\end{defn}
Clearly, a proper ideal is weakly $(m,n)-$closed $\delta-$primary for $1\leq~m\leq~n$; so we usually assume that $1\leq~n<m$.
\begin{thm}
Let $R$ be a commutative ring, $I$ be a proper ideal of $R$, $\delta$ be an expansion function of $Id(R)$, and $m$ and $n$ be positive integers.  Then
\begin{enumerate}
\item[(1)] $I$ is a weakly semi$-n-$absorbing $\delta-$primary ideal of $R$ if and only if I is a weakly $(n+1,n)-$closed $\delta-$primary ideal of $R$.\\
\item[(2)] If $I$ is a weakly $n-$absorbing $\delta-$primary ideal of $R$, then $I$ is a weakly semi$-n-$absorbing $\delta-$primary ideal of $R$.\\
\item[(3)] If $I$ is a weakly $(m,n)-$closed $\delta-$primary ideal of $R$, then $I$ is a weakly $(m,k)-$closed $\delta-$primary ideal of $R$ for every positive integer $k\geq~n$.\\
\item[(4)] A weakly $n-$absorbing ideal of $R$ is a weakly $(m,n)-$closed $\delta-$primary ideal of $R$ for every positive integer $m$.
\end{enumerate}
\end{thm}
\begin{proof}
Follows directly from the definitions.
\end{proof}
In the following example, we give some expansion functions of ideals of a ring $R$.
\begin{Exa}
\begin{enumerate}
\item[(1)] The identity function $\delta_I$, where $\delta_I(I)=I$ for every $I\in~Id(R)$, is an expansion function of ideals of $R$.\\
\item[(2)] For each ideal $I,$ define $\delta_{\sqrt{I}}(I)=\sqrt{I}$.  Then $\delta_{\sqrt{I}}$ is an expansion function of ideals of $R$.
\end{enumerate}
\end{Exa}
\begin{rem}
\begin{enumerate}
\item[(1)] Let $R$ be a commutative ring, $\delta_I$ be an expansion function of $Id(R)$ and $m$ and $n$ positive be integers, then a proper ideal $I$ of $R$ is a weakly $(m,n)-$closed $\delta_I-$primary ideal of $R$ if and only if $I$ is a weakly $(m,n)-$closed ideal of $R$.\\
\item[(2)] It is clear that any $(m,n)-$closed $\delta-$primary ideal of $R$ is weakly $(m,n)-$closed $\delta-$primary ideal.  The converse need not hold.\\
\item[(3)] A weakly $(m,n)-$closed $\delta-$primary ideal does not be weakly $(\acute{m},n)-$closed $\delta-$primary for $\acute{m}<m$.
\end{enumerate}
\end{rem}
\begin{Exa}
Let $R=\mathbb{Z}_8$.  Then $I=(0),$ the zero ideal is clearly a weakly $(3,1)-$closed $\delta_{\sqrt{I}}-$primary ideal of $\mathbb{Z}_8$.  However, $I$ is not $(3,1)-$closed $\delta_{\sqrt{I}}-$primary ideal of $\mathbb{Z}_8$ since $2^3=0\in~I$ and $2\not\in~I$.  More generally, $I=(0)$ is a weakly $(n+1,n)-$closed $\delta_{\sqrt{I}}-$primary ideal of $\mathbb{Z}_{2^{n+1}}$, by definition, but it is easy to see that $I$ is not an $(n+1,n)-$closed $\delta_{\sqrt{I}}-$primary ideal of $\mathbb{Z}_{2^{n+1}}$.
\end{Exa}
\begin{Exa}
Let $R=\mathbb{Z}_8$ and $\delta$ an expansion function of $Id(R)$ and $I=\{0,4\}$.  Then $I$ is weakly $(3,1)-$closed $\delta-$primary ideal of $\mathbb{Z}_8$, since $r^3=0$ for every nonunit $r\in~R$.  However, $I$ is not weakly $(2,1)-$closed $\delta-$primary ideal since $0\neq2^2=4\in~I$ and $2\not\in~I$.
\end{Exa}
In the following example, we have a weakly $(m,n)-$closed $\delta-$primary ideal of $R$, that is neither a weakly $(m,n)-$closed ideal of $R$ nor an $(m,n)-$closed $\delta-$primary ideal of $R$.
\begin{Exa}
Let $A=\mathbb{Z}_{2^m}[\{X_m\}_{m\in\mathbb{N}}]$ and $I=(X_{m-1}^mA)$ be an ideal of $A$.  Let $R=A/I$ and define $\delta:Id(R)\longrightarrow~Id(R)$ such that $\delta(K)=K+(xA+I)/I$.  It is clear that $\delta$ is an expansion function of $Id(R)$.  Let $J=(X_m^mA)/I$.  We show that $J$ is not a weakly $(m,n)-$closed ideal of $R$.  Now, in the ring $R$, we have $0\neq~X_m^m+I\in~J$, but $X_m^n+I\not\in~J$ for every positive integers $1\leq~n<m$.  Thus, $J$ is not a weakly $(m,n)-$closed ideal of $R$.  We show that $J$ is not an $(m,n)-$closed $\delta-$primary ideal of $R$.  Let $x=2+I\in~R$.  Then $x^m=0+I\in~J$, but $x^n\not\in\delta(J)$ for every positive integers $1\leq~n<m$.  Thus, $J$ is not an $(m,n)-$closed $\delta-$primary ideal of $R$.  By the construction of $\delta$, one can easily seen that $J$ is a weakly $(m,n)-$closed $\delta-$primary ideal of $R$.
\end{Exa}
\begin{thm}
Let $R$ be a commutative ring, $\delta$ and $\gamma$ be expansion functions of $Id(R)$ with $\delta(I)\subseteq\gamma(I)$, $m$ and $n$ be integers with $1\leq~n<m$, and $I$ be a proper ideal of $R$.
\begin{enumerate}
\item[(1)] If $I$ is a weakly $(m,n)-$closed $\delta-$primary ideal of $R$, then $I$ is a weakly $(m,n)-$closed $\gamma-$primary ideal of $R$.
\item[(2)] If $\delta(I)$ is a weakly $(m,n)-$closed ideal of $R$, then $I$ is a weakly $(m,n)-$closed $\delta-$primary ideal of $R$.
\end{enumerate}
\end{thm}
\begin{proof}
\begin{enumerate}
\item[(1)] Since $\delta(I)\subseteq\gamma(I)$ and $I$ is a weakly $(m,n)-$closed $\delta-$primary ideal of $R$, then the claim follows.
\item[(2)]  Let $0\neq~a^m\in~I\subseteq\delta(I)$ for some $a\in~R$.  Since $\delta(I)$ is a weakly $(m,n)-$closed ideal of $R$, then $a^n\in\delta(I)$.  Therefore, $I$ is a weakly $(m,n)-$closed $\delta-$primary ideal of $R$.
\end{enumerate}
\end{proof}
Let $R$ be a commutative ring, $\delta_1$ and $\delta_2$ are expansion functions of $Id(R)$.  Let $\delta:Id(R)\longrightarrow~Id(R)$ such that $\delta(I)=(\delta_1\circ\delta_2)(I)=\delta_1(\delta_2(I))$.  Then, $\delta$ is an expansion function of ideals of $R$, such a function is denoted by $\delta_\circ$.
\begin{thm}
Let $R$ be a commutative ring, $\delta$ and $\gamma$ be expansion functions of $Id(R)$, $m$ and $n$ positive be integers, $I$ be a proper ideal of $R$, and $\delta(0)$ be an $(m,n)-$closed $\gamma-$primary ideal of $R$.  Then, $I$ is a weakly $(m,n)-$closed $\gamma\circ\delta-$primary ideal of $R$ if and only if $I$ is an $(m,n)-$closed $\gamma\circ\delta-$primary ideal of $R$
\end{thm}
\begin{proof}
Assume that $I$ is a weakly $(m,n)-$closed $\gamma\circ\delta-$primary ideal of $R$.  Assume that $a^m\in~I$ for some $a\in~R$.  If $0\neq~a^m\in~I$, then $a^n\in\gamma(\delta(I))$.  Hence, assume that $a^m=0$.  Since $\delta(0)$ is an $(m,n)-$closed $\gamma-$primary ideal of $R$, we conclude that $a^n\in\gamma(\delta(0))$.  Since $\gamma(\delta(0))\subseteq\gamma(\delta(I))$, we conclude that $a^n\in\gamma(\delta(0))\subseteq\gamma(\delta(I))$.  Therefore, $I$ is an $(m,n)-$closed $\gamma\circ\delta-$primary ideal of $R$.  The converse is clear.
\end{proof}
\begin{thm}
Let $R$ be a commutative ring, $\delta$ and $\gamma$ be expansion functions of $Id(R)$, $m$ and $n$ be positive integers, $I$ be a proper ideal of $R$, and $\gamma(I)$ be a weakly $(m,n)-$closed $\delta-$primary ideal of $R$.  Then, $I$ is a weakly $(m,n)-$closed $\delta\circ\gamma-$primary ideal of $R$.
\end{thm}
\begin{proof}
Assume that $0\neq~a^m\in~I\subseteq\gamma(I)$ for some $a\in~R$.  Since $\gamma(I)$ is a weakly $(m,n)-$closed $\delta-$primary ideal of $R$.  Then, $a^n\in\delta(\gamma(I))$.  Therefore, $I$ is a weakly $(m,n)-$closed $\delta\circ\gamma-$primary ideal of $R$.
\end{proof}
\begin{defn}\cite{m}
Let $f:R\longrightarrow~A$ be a ring homomorphism and $\delta,~\gamma$ expansion functions of $Id(R)$ and $Id(A)$ respectively.  Then, $f$ is called a $\delta\gamma-$homomorphism if $\delta(f^{-1}(I))=f^{-1}(\gamma(I))$ for all ideals $I$ of $A$.
\end{defn}
\begin{rem}
$(1)$ If $\gamma_r$ is a radical operation on a commutative ring $A$ and $\delta_r$ is a radical operation on a commutative ring $R$, then any homomorphism from $R$ to $A$ is an example of $\delta_r\gamma_r-$homomorphism.  Also, if $f$ is a $\delta\gamma-$epimorphism and $I$ is an ideal of $R$ containing $Ker(f)$, then $\gamma(f(I))=f(\delta(I))$. In particular, if $f$ is a $\delta\gamma-$ring$-$isomorphism, then $f(\delta(I))=\gamma(f(I))$ for every ideal $I$ of $R$.\\
$(2)$ Let $R$ be a commutative ring, $\delta$ be an expansion function of $Id(R)$ and $I$ be a proper ideal of $R$.  The function $\delta_q:R/I\longrightarrow~R/I$ defined by $\delta_q(J/I)=\delta(J)/I$ for ideals $I\subseteq~J$, becomes an expansion function of $R/I$.
\end{rem}
\begin{thm}\label{3}
Let $R$ and $A$ be commutative rings, $m$ and $n$ be positive integers, and $f:R\longrightarrow~A$ be a $\delta\gamma-$homomorphism, where $\delta$ is an expansion function of $Id(R)$ and $\gamma$ is an expansion function of $Id(A)$.\\
\begin{enumerate}
\item[(1)] If $f$ is injective and $J$ is a weakly $(m,n)-$closed $\gamma-$primary ideal of $A$, then $f^{-1}(J)$ is a weakly $(m,n)-$closed $\delta-$primary ideal of $R$.\\
\item[(2)] If $f$ is surjective and $I$ is a weakly $(m,n)-$closed $\delta-$primary ideal of $R$ containing $Ker(f)$, then $f(I)$ is a weakly $(m,n)-$closed $\gamma-$primary ideal of $A$.
\end{enumerate}
\end{thm}
\begin{proof}
\begin{enumerate}
\item[(1)] Assume that $J$ is a weakly $(m,n)-$closed $\gamma-$primary ideal of $A$ and $0\neq~a^m\in~f^{-1}(J)$ for some $a\in~R$.  Since, $Ker(f)\neq~0$, we get $0\neq~f(a^m)=[f(a)]^m\in~J$.  By our assumption, we conclude that $[f(a)]^n\in\gamma(J)$.  Thus, $a^n\in~f^{-1}(\gamma(J))$.  Since $\delta(f^{-1}(J))=f^{-1}(\gamma(J))$, we get $a^n\in\delta(f^{-1}(J))$.  Therefore, $f^{-1}(J)$ is a weakly $(m,n)-$closed $\delta-$primary ideal of $R$.
\item[(2)] Assume that $I$ is a weakly $(m,n)-$closed $\delta-$primary ideal of $R$ and $0\neq~b^m\in~f(I)$ for some $b\in~A$.  Since $f$ is epimorphism, we have $f(a^m)=b^m$ for some $a\in~R$ and $0\neq~f(a^m)=b^m\in~f(I)$.  Since $Ker(f)\subseteq~I$, we have $0\neq~a^m\in~I$.  As $I$ is an $(m,n)-$closed $\delta-$primary ideal of $R$, we have $a^n\in\delta(I)$.  Then, we have $b^n\in~f(\delta(I))\subseteq\gamma(f(I))$.  Therefore, $f(I)$ is a weakly $(m,n)-$closed $\gamma-$primary ideal of $A$.
\end{enumerate}
\end{proof}
\begin{thm}\label{80}
Let $R$ be a commutative ring, $\delta$ be an expansion function of $Id(R)$, $m$ and $n$ be positive integers, and $I\subseteq~J$ be prober ideals of $R$.
\begin{enumerate}
\item[(1)] If $J$ is a weakly $(m,n)-$closed $\delta-$primary ideal of $R$, then $J/I$ is a weakly $(m,n)-$closed $\delta_q-$primary ideal of $R/I$.
\item[(2)] If $I$ is an $(m,n)-$closed $\delta-$primary ideal of $R$ and $J/I$ is a weakly $(m,n)-$closed $\delta_q-$primary ideal of $R/I$, then $J$ is an $(m,n)-$closed $\delta-$primary ideal of $R$.
\item[(3)] If $I$ is a weakly $(m,n)-$closed $\delta-$primary ideal of $R$ and $J/I$ is a weakly $(m,n)-$closed $\delta_q-$primary ideal of $R/I$, then $J$ is a weakly $(m,n)-$closed $\delta-$primary ideal of $R$.
\end{enumerate}
\end{thm}
\begin{proof}
\begin{enumerate}
\item[(1)] Follows directly from Theorem \ref{3} $(2)$.\\
\item[(2)] Let $a^m\in~J$ for $a\in~R$.  If $a^m\in~I$, then $a^n\in\delta(I)\subseteq\delta(J)$.  Assume that $a^m\not\in~I$.  Then, we have $I\neq(a+I)^m\in~J/I$.  Since $J/I$ is a weakly $(m,n)-$closed $\delta_q-$primary ideal of $R/I$.  Then, $(a+I)^n=a^n+I\in\delta_q(J/I)$.  Thus, $a^n\in\delta(J)$.  Therefore, $J$ is an $(m,n)-$closed $\delta-$primary ideal of $R$.\\
\item[(3)] Let $0\neq~a^m\in~J$ for $a\in~R$.  Then by a similar argument as in part (2), one can show that $J$ is a weakly $(m,n)-$closed $\delta-$primary ideal of $R$.
\end{enumerate}
\end{proof}
\begin{Exa}
Let $D$ be an integral domain, $I$ be a proper ideal of $D$, $m,~n$ be positive integers with $1\leq~n<m$, and $\delta$ be an expansion function of $Id(D)$.  Then $<I,X>$ is a weakly $(m,n)-$closed $\delta-$primary ideal of $R[x]$ if and only if $I$ is a weakly $(m,n)-$closed $\delta-$primary ideal of $R$.  This is hold by Theorem \ref{80} $(1),~(2)$, since $<I,X>/<x>\approx~I$ in $R[x]/<x>\approx~R$.
\end{Exa}
\begin{thm}\label{30}
Let $R$ be a commutative ring, $I$ be a proper ideal of $R$, $m$ and $n$ be positive integers, $S\subseteq~R\backslash\{0\}$ be a multiplicative set with $I\cap~S=\phi$, and $\delta_S$ be an expansion function of $Id(R_S)$ such that $\delta_S(I_S)=(\delta(I))_S$ where $\delta$ is an expansion function of $Id(R)$.  If $I$ is a weakly $(m,n)-$closed $\delta-$primary ideal of $R$, then $I_S$ is a weakly $(m,n)-$closed $\delta_S-$primary ideal of $R_S$.
\end{thm}
\begin{proof}
Let $0\neq~a^m\in~I_S$ for $a\in~R_S$.  Then, $a=\frac{r}{s}$ for some $r\in~R$ and $s\in~S$, and thus $0\neq~a^m=\frac{r^m}{s^m}=\frac{i}{g}$ for some $i\in~I$ and $g\in~S$. Hence, $0\neq~r^mgz=s^miz\in~I$ for some $z\in~S$, and thus $0\neq~(rgz)^m\in~I$.  Hence, $(rgz)^n\in\delta(I)$ since $I$ is weakly $(m,n)-$closed $\delta-$primary ideal of $R$, and thus $a^n=\frac{r^n}{s^n}=\frac{r^ng^nz^n}{s^ng^nz^n}\in~I_S\subseteq(\delta(I))_S=\delta_S(I_S)$.  Therefore, $I_S$ is a weakly $(m,n)-$closed $\delta_S-$primary ideal of $R_S$.
\end{proof}
\begin{cor}
Let $R$ be a commutative ring, $I$ be a proper ideal of $R$, $m$ and $n$ be positive integers, $P\subseteq~R\backslash\{0\}$ be a multiplicative set with $I\cap~P=\phi$, and $\delta_P$ be an expansion function of $Id(R_P)$ such that $\delta_P(I_P)=(\delta(I))_P$ where $\delta$ is an expansion function of $Id(R)$.  Then the following statements are equivalent.
\begin{enumerate}
\item[(1)] $I$ is a weakly $(m,n)-$closed $\delta-$primary ideal of $R$.\\
\item[(2)] $I_P$ is a weakly $(m,n)-$closed $\delta_P-$primary ideal of $R_P$ for every prime (or maximal) ideal $P$ of $R$.
\end{enumerate}
\end{cor}
\begin{proof}
$(1)\Longrightarrow(2)$ Follows directly from Theorem \ref{30}.\\
$(2)\Longrightarrow(1)$ Let $0\neq~a^m\in~I$ for $a\in~R$, consider the ideal $J=\{r\in~R:ra^n\in\delta(I)\}$ of $R$ and $P$ be a prime ideal of $R$ with $I\subseteq~P$.  Then, $0\neq(\frac{a}{1})^m\in~I_P$.  Thus, $(\frac{a}{1})^n\in\delta_P(I_P)$ since $I_P$ is a weakly $(m,n)-$closed $\delta_P-$primary ideal of $R_P$.  Thus, $xa^n\in\delta(I)$ for some $x\in~R\backslash~P$.  Thus, $J\nsubseteq~P$.  In fact $J\nsubseteq~L$ for every prime ideal $L$ of $R$ with $I\nsubseteq~L$.  Hence, $J=R$ and thus $a^n\in\delta(I)$.  Therefore,  $I$ is a weakly $(m,n)-$closed $\delta-$primary ideal of $R$.
\end{proof}
Recall from \cite{m} that an expansion function $\delta$ of $Id(R)$ satisfies the finite intersection property if $\delta(I_1\cap...\cap~I_n)=\delta(I_1)\cap...\cap\delta(I_n)$ for some ideals $I_1,...,I_n$ of the commutative ring $R$.\\
Note that the radical operation on ideals of a commutative ring is an example of an expansion function satisfying the finite intersection property.\\
\begin{thm}\label{11}
Let $R$ be a commutative ring, $\delta$ be an expansion function of $Id(R)$ satisfying the finite intersection property, $m$ and $n$ be positive integers, and $I_1,...,I_k$ be proper ideals of $R$.  If $I_1,...,I_k$ are weakly $(m,n)-$closed $\delta-$primary ideal of $R$, and $P=\delta(I_j)$ for all $j\in\{1,...,k\}$, then $I_1\cap...\cap~I_k$ is a weakly $(m,n)-$closed $\delta-$primary ideal of $R$.
\end{thm}
\begin{proof}
It is clear.
\end{proof}
\begin{defn}
Let $R$ be a commutative ring, $\delta$ be an expansion function of $Id(R)$, $m$ and $n$ be integers with $1\leq~n<m$, and $I$ be a weakly $(m,n)-$ closed $\delta-$primary ideal of $R$.  Then $a\in~R$ is called a $\delta-(m,n)-$unbreakable-zero element of $I$ if $a^m=0$ and $a^n\not\in\delta(I)$.  (Thus, $I$ has a $\delta-(m,n)-$unbreakable-zero element if and only if $I$ is not $(m,n)-$closed $\delta-$primary ideal of $R$).
\end{defn}
\begin{thm}\label{81}
Let $R$ be a commutative ring, $\delta$ be an expansion function of $Id(R)$, $m$ and $n$ be integers with $1\leq~n<m$, and $I$ be a weakly $(m,n)-$ closed $\delta-$primary ideal of $R$.  If $x$ is a $\delta-(m,n)-$unbreakable-zero element of $I$, then $(x+i)^m=0$ for every $i\in~I$.
\end{thm}
\begin{proof}
Assume that $(x+i)^m\neq0$ for some $i\in~I$.  Then by Binomial Theorem, we have $(x+i)^m=x^m+mx^{m-1}i+...+i^m=0+mx^{m-1}i+...+i^m\in~I$.  Since $I$ is a weakly $(m,n)-$ closed $\delta-$primary ideal of $R$, then $(x+i)^n\in\delta(I)$, a contradiction.  Therefore, $(x+i)^m=0$ for every $i\in~I$.
\end{proof}
\begin{thm}
Let $R$ be a commutative ring, $\delta$ be an expansion function of $Id(R)$, $m$ and $n$ be positive integers, and $I$ be a weakly $(m,n)-$ closed $\delta-$primary ideal of $R$ that is not $(m,n)-$ closed $\delta-$primary.  Then $I\subseteq~Nil(R)$.  Furthermore, if $char(R)=m$ is prime, then $x^m=0$ for every $x\in~I$.
\end{thm}
\begin{proof}
Since $I$ is a weakly $(m,n)-$ closed $\delta-$primary ideal of $R$ that is not $(m,n)-$ closed $\delta-$primary, then $I$ has a $\delta-(m,n)-$unbreakable-zero element say $a$.  Let $x\in~I$.  Then, $a^m=0$ and $(a+x)^m=0$ by Theorem \ref{81}.  Thus, $a,~a+x\in~Nil(R)$.  Thus, $x=(a+x)-a\in~Nil(R)$.  Therefore, $I\subseteq~Nil(R)$.  Now, assume that $char(R)=m$ is prime.  Then, it is clear that $0=(a+x)^m=a^m+x^m=x^m$.  Therefore, $x^m=0$ for every $x\in~I$.
\end{proof}
\begin{lem}
Let $R$ be a commutative ring, $\delta$ be an expansion function of $Id(R)$, $n$ be a positive integer and $a\in~J(R)$. Then the ideal $I=a^{n+1}R$ is a weakly semi$-n-$absorbing $\delta-$primary ideal of $R$ if and only if $a^{n+1}=0$.
\end{lem}
\begin{proof}
Assume that $a^{n+1}=0$.  Then, $I=(0)$ and thus $I$ is a weakly semi$-n-$absorbing $\delta-$primary ideal of $R$ by definition.\\
Conversely; assume that $I$ is a weakly semi$-n-$absorbing $\delta-$primary ideal of $R$ and $a^{n+1}\neq~0$.  Since $a^{n+1}\in~I\backslash\{0\}$, then $a^n\in\delta(I)$.  Set $x=a^n\in~I$.  Then, $x=xar$ for some $r\in~R$.  Thus $x(1-ar)=0$.  Since $ar\in~J(R)$, then $1-ar$ is a unit of $R$.  Thus, $x=0$ and then $a^{n+1}=0$, a contradiction.  Therefore, $a^{n+1}=0$.
\end{proof}
\begin{thm}
Let $D$ be an integral domain, $\delta$ be an expansion function of $Id(D)$, $I=p^kD$ be a principal ideal of $D$, where $p$ is a prime element of $D$ and $k$ a positive integer.  Let $m$ be a positive integer such that $m<k$, and write $k=ma+b$ for some integers $a,b$, where $a\geq1$ and $0\leq~b<m$.  If $I/p^cD$ is a weakly $(m,n)-$closed $\delta-$primary ideal of $D/p^cD$ that is not $(m,n)-$closed $\delta-$primary ideal for positive integers $n<m$ and $c\geq~k+1$, then $b\neq0$, $k+1\leq~c\leq~m(a+1)$ and $n(a+1)<k$.
\end{thm}
\begin{proof}
Assume that $I/p^cD$ is a weakly $(m,n)-$closed $\delta-$primary ideal of $D/p^cD$ that is not $(m,n)-$closed $\delta-$primary for positive integers $n<m$ and $c\geq~k+1$.  It is clear that $b\neq0$, for if $b=0$, then $0\neq(p^a)^m+p^cD\in~I/p^cD$, but $(p^a)^n+p^cD\not\in~I/p^cD$.  Since $a+1$ is the smallest positive integer such that $(p^{(a+1)})^m+p^cD\in~I/p^cD$ and $I/p^cD$ is not $(m,n)-$closed $\delta-$primary ideal of $D/p^cD$, we have $0=(p^{(a+1)})^m+p^cD\in~I/p^cD$ and $(p^{(a+1)})^n+p^cD\not\in\delta(I/p^cD)$.  Thus, $(p^{(a+1)})^n+p^cD\not\in~I/p^cD$ and therefore, $n(a+1)<k$ and $k+1\leq~c\leq~m(a+1)$.
\end{proof}
Recall that an ideal of $R_1\times~R_2$ has the form $I_1\times~I_2$ for some ideals $I_1$ of $R_1$ and $I_2$ of $R_2$, where $R_1$ and $R_2$ are commutative rings.  Let $R=R_1\times...\times~R_k$, where $R_i$ is a commutative ring with nonzero identity and $\delta_i$ be an expansion function of $Id(R_i)$ for each $i\in\{1,2,...,k\}$.  Let $\delta_\times$ be a function of $Id(R)$, which is defined by $\delta_\times(I_1\times~I_2\times...\times~I_k)=\delta_1(I_1)\times\delta_2(I_2)\times...\times\delta_k(I_k)$ for each ideal $I_i$ of $R_i$ were $i\in\{1,2,...,k\}$.  Then $\delta_\times$ is an expansion function of $Id(R)$.  Note that every ideal of $R$ is of the form $I_1\times~I_2\times...\times~I_k$, where each ideal $I_i$ is an ideal of $R_i$, $i\in\{1,2,...,k\}$.
In the next results, we characterize weakly $(m,n)-$closed $\delta-$primary ideals of $R_1\times...\times~R_k$.
\begin{thm}\label{4}
Let $R=R_1\times~R_2$, where $R_1$ and $R_2$ are commutative rings, $\delta_i$ be an expansion function of $Id(R_i)$ for each $i\in\{1,2\}$, $I_1$ be a proper ideal of $R_1$, $I_2$ be a proper ideal of $R_2$, and $m$ and $n$ be positive integers with $1\leq~n<m$.  Then the following statements are equivalent.
\begin{enumerate}
\item[(1)] $I_1\times~R_2$ is a weakly $(m,n)-$closed $\delta_\times-$primary ideal of $R$.\\
\item[(2)] $I_1$ is an $(m,n)-$closed $\delta_1-$primary ideal of $R_1$.\\
\item[(3)] $I_1\times~R_2$ is an $(m,n)-$closed $\delta_\times-$primary ideal of $R$.
\end{enumerate}
\end{thm}
\begin{proof}
$(1)\Longrightarrow(2)$ $I_1$ is a weakly $(m,n)-$closed $\delta_1-$primary ideal of $R_1$ by Theorem \ref{3} (2).  If $I_1$ is not an $(m,n)-$closed $\delta_1-$primary ideal of $R_1$, then $I_1$ has a $\delta_1-(m,n)-$unbreakable-zero element $x$.  Thus, $(0,0)\neq(x,1)^m\in~I_1\times~R_2$, but $(x,1)^n\not\in\delta_\times(I_1\times~R_2)$, a contradiction.  Hence, $I_1$ is an $(m,n)-$closed $\delta_1-$primary ideal of $R_1$.\\
$(2)\Longrightarrow(3)$  It is clear.  (Analogue to the proof of \cite{a}, Theorem 2.12).\\
$(3)\Longrightarrow(1)$  It is clear by definition.
\end{proof}
\begin{thm}
Let $R=R_1\times~R_2$, where where $R_1$ and $R_2$ are commutative rings, $\delta_i$ be an expansion function of $Id(R_i)$ for each $i\in\{1,2\}$, $I$ be a proper ideal of $R$, and $m$ and $n$ be positive integers with $1\leq~n<m$.  Then the following statements are equivalent.
\begin{enumerate}
\item[(1)] $I$ is a weakly $(m,n)-$closed $\delta_\times-$primary ideal of $R$ that is not $(m,n)-$closed $\delta_\times-$primary.\\
\item[(2)] $I=J_1\times~J_2$ for some proper ideals $J_1$ of $R_1$ and $J_2$ of $R_2$ such that either\\
 (a) $J_1$ is a weakly $(m,n)-$closed $\delta_1-$primary ideal of $R_1$ that is not $(m,n)-$closed $\delta_1-$primary, $y^m=0$ whenever $y^m\in~J_2$ for some $y\in~R_2$ (in particular, $j^m=0~\forall~j\in~J_2$), and if $0\neq~x^m\in~J_1$ for some $x\in~R_1$, then $J_2$ is an $(m,n)-$closed $\delta_2-$primary ideal of $R_2$, or\\
 (b) $J_2$ is a weakly $(m,n)-$closed $\delta_2-$primary ideal of $R_2$ that is not $(m,n)-$closed $\delta_2-$primary, $y^m=0$ whenever $y^m\in~J_1$ for some $y\in~R_1$ (in particular, $j^m=0~\forall~j\in~J_1$), and if $0\neq~x^m\in~J_2$ for some $x\in~R_2$, then $J_1$ is an $(m,n)-$closed $\delta_1-$primary ideal of $R_1$.
\end{enumerate}
\end{thm}
\begin{proof}
$(1)\Longrightarrow(2)$ Since $I$ is not an $(m,n)-$closed $\delta_\times-$primary ideal of $R$.  Then, by Theorem \ref{4}, we have $I=J_1\times~J_2$, where $J_1$ is a proper ideal of $R_1$ and $J_2$ is a proper ideal of $R_2$.  Since $I$ is not an $(m,n)-$closed $\delta_\times-$primary ideal of $R$, we have either $J_1$ is a weakly $(m,n)-$closed $\delta_1-$primary ideal of $R_1$ that is not $(m,n)-$closed $\delta_1-$primary or $J_2$ is a weakly $(m,n)-$closed $\delta_2-$primary ideal of $R_2$ that is not $(m,n)-$closed $\delta_2-$primary.  Suppose that $J_1$ is a weakly $(m,n)-$closed $\delta_1-$primary ideal of $R_1$ that is not $(m,n)-$closed $\delta_1-$primary.  Thus, $J_1$ has a $\delta_1-(m,n)-$unbreakable-zero element $a$.  Suppose that $y^m\in~J_2$ for some $y\in~R_2$.  Since $a$ is a $\delta_1-(m,n)-$unbreakable-zero element of $J_1$ and $(a,y)^m\in~I$, we have $(a,y)^m=(0,0)$.  Thus, $y^m=0$ (in particular, $j^m=0~\forall~j\in~J_2$).  Now, suppose that $0\neq~x^m\in~J_1$ for some $x\in~R_1$.  Let $y\in~R_2$ be such that $y^m\in~J_2$.  Then, $(0,0)\neq(x,y)^m\in~I$.  Thus, $y^n\in\delta_2(J_2)$, and hence $J_2$ is an $(m,n)-$closed $\delta_2-$primary ideal of $R_2$.  Similarly, if $J_2$ is a weakly $(m,n)-$closed $\delta_2-$primary ideal of $R_2$ that is not $(m,n)-$closed $\delta_2-$primary, then $y^m=0$ whenever $y^m\in~J_1$ for some $y\in~R_1$ (in particular, $j^m=0~\forall~j\in~J_1$), and if $0\neq~x^m\in~J_2$ for some $x\in~R_2$, then $J_1$ is an $(m,n)-$closed $\delta_1-$primary ideal of $R_1$.\\
$(2)\Longrightarrow(1)$ Assume that $J_1$ is a weakly $(m,n)-$closed $\delta_1-$primary ideal of $R_1$ that is not $(m,n)-$closed $\delta_1-$primary, $y^m=0$ whenever $y^m\in~J_2$ for some $y\in~R_2$ (in particular, $j^m=0~\forall~j\in~J_2$), and if $0\neq~x^m\in~J_1$ for some $x\in~R_1$, then $J_2$ is an $(m,n)-$closed $\delta_2-$primary ideal of $R_2$.  Let $a$ be a $\delta_1-(m,n)-$unbreakable-zero element of $J_1$.  Then, $(a,0)$ is a $\delta_\times-(m,n)-$unbreakable-zero element of $I$.  Thus, $I$ is not an $(m,n)-$closed $\delta_\times-$primary ideal of $R$.  Now, suppose that $(0,0)\neq(x,y)^m=(x^m,y^m)\in~I$ for some $x\in~R_1$ and some $y\in~R_2$.  Then, $(0,0)\neq(x,y)^m=(x^m,0)\in~I$ and $0\neq~x^m\in~J_1$.  Since $J_1$ is a weakly $(m,n)-$closed $\delta_1-$primary ideal of $R_1$ and $J_2$ is an $(m,n)-$closed $\delta_2-$primary ideal of $R_2$, we have $(x,y)^n\in\delta_\times(I)$.  Similarly, suppose that $J_2$ is a weakly $(m,n)-$closed $\delta_2-$primary ideal of $R_2$ that is not $(m,n)-$closed $\delta_2-$primary, $y^m=0$ whenever $y^m\in~J_1$ for some $y\in~R_1$ (in particular, $j^m=0~\forall~j\in~J_1$), and if $0\neq~x^m\in~J_2$ for some $x\in~R_2$, then $J_1$ is an $(m,n)-$closed $\delta_1-$primary ideal of $R_1$.  Then again, $I$ is a weakly $(m,n)-$closed $\delta_\times-$primary ideal of $R$ that is not $(m,n)-$closed $\delta_\times-$primary.
\end{proof}
Let $R$ be a commutative ring, $\delta$ be an expansion function of $Id(R)$ and $M$ be an $R-$module.  As in \cite{l}, the trivial ring extension of $R$ by $M$ (or the idealization of $M$ over $R$) is defined by $R(+)M(other~notation~is~R\propto~M)=\{(a,b):a\in~R,~b\in~M\}$ is a commutative ring with identity $(1,0)$ under addition defined by $(a,b)+(c,d)=(a+c,b+d)$ and multiplication defined by $(a,b)(c,d)=(ac,ad+bc)$ for each $a,c\in~R$ and $b,d\in~M$.  Note that $(\{0\}(+)M)^2=\{0\}$, so $\{0\}(+)M\subseteq~Nil(R(+)M)$.\\
We define a function $\delta_{(+)}:Id(R(+)M)\longrightarrow~Id(R(+)M)$ such that $\delta_{(+)}(I(+)N)=\delta(I)(+)M$ for every ideal $I$ of $R$ and every submodule $N$ of $M$.  Then $\delta_{(+)}$ is an expansion function of ideals of $R(+)M$.
\begin{thm}\label{34}
Let $R$ be a commutative ring, $\delta$ be an expansion function of $Id(R)$, $m$ and $n$ be integers with $1\leq~n<~m$, $I$ be a prober ideal of $R$, and $M$ be an $R-$module.  Then the following statements are equivalent.\\
\begin{enumerate}
\item[(1)] $I(+)M$ is a weakly $(m,n)-$closed $\delta_{(+)}-$primary ideal of $R(+)M$ that is not $(m,n)-$closed $\delta_{(+)}-$primary.\\
\item[(2)] $I$ is a weakly $(m,n)-$closed $\delta-$primary ideal of $R$ that is not $(m,n)-$closed $\delta-$primary and $m(x^{m-1}M)=0$ for every $(m,n)-$unbrekable-zero element $x$ of $I$.
\end{enumerate}
\end{thm}
\begin{proof}
$(1)\Longrightarrow(2)$ Assume that $0\neq~r^m\in~I$ for some $r\in~R$.  Then, $(0,0)\neq(r,0)^m=(r^m,0)\in~I(+)M$.  Thus, $(r,0)^n=(r^n,0)\in~I(+)M$; so $r^n\in\delta(I)$.  Thus, $I$ is a weakly $(m,n)-$closed $\delta-$primary ideal of $R$.  Since $I(+)M$ is not $(m,n)-$closed $\delta_{(+)}-$primary ideal of $R(+)M$, we have $I(+)M$ and hence $I$, has a $\delta-(m,n)-$unbrekable-zero element; so $I$ is not an $(m,n)-$closed $\delta-$primary ideal of $R$.  Let $x$ be a $\delta-(m,n)-$unbrekable-zero element of $I$ and $a\in~M$.  Then, $(x,a)^m=(x^m,m(x^{m-1}a))\in~I(+)M$.  Since $a^n\not\in\delta(I)$, we have $(x,a)^m=(x^m,m(x^{m-1}a))=(0,0)$.  Thus $m(x^{m-1}M)=0$.\\
$(2)\Longrightarrow(1)$ Since $I$ is a weakly $(m,n)-$closed $\delta-$primary ideal of $R$ that is not $(m,n)-$closed $\delta-$primary, we have $I$  has a $\delta-(m,n)-$unbrekable-zero element $x$.  Thus, $(x,0)$ is a $\delta_{(+)}-(m,n)-$unbrekable-zero element of $I(+)M$.  Thus, $I(+)M$ is not an $(m,n)-$closed $\delta_{(+)}-$primary ideal of $R(+)M$.  Assume that $(0,0)\neq(r,y)^m=(r^m,m(r^{m-1}y))\in~I(+)M$.  Then, $r$ is not a $\delta-(m,n)-$unbrekable-zero element of $I$ by hypothesis.  Hence, $(r^n,n(r^{n-1}y)=(r,y)^n\in\delta(I(+)M)$.  Therefore, $I(+)M$ is a weakly $(m,n)-$closed $\delta_{(+)}-$primary ideal of $R(+)M$ that is not $(m,n)-$closed $\delta_{(+)}-$primary.
\end{proof}
\section{weakly $(m,n)-$closed $\delta-$primary ideals in amalgamated algebra}
Let $f:A\rightarrow B$ be a ring homomorphism and $J$ be an ideal of $B.$ We
define the following subring of $A\times B:$
$$A\bowtie^{f}J=\left\{(a,f(a)+j)|~a\in A,j\in J\right\}$$

called \emph{the amalgamation of $A$ with $B$ along $J$ with
respect to $f.$} This construction is a generalization of the
amalgamated duplication of a ring along an ideal introduced and
studied in \cite{hh,hhh}. If $A$ is a commutative ring with
unity, and $I$ be a ideal of $A$, the amalgamated duplication of
$A$ along the ideal $I$, coincides with $A\bowtie^{id}I,$ where $id$ is the identity map $id:=id_A:A\rightarrow~A$, we have
$$A\bowtie I=\{(a,a+i)|~a\in A,i\in I\}.$$

The interest of amalgamation resides, partly, in its ability to
cover serval constructions in commutative rings, including
specially trivial extension(also called Nagat's
idealizations \cite{hhhh}). Moreover, other constructions
($A+XB[X],A+XB[[X]]$, and the $D+M$ construction) can be studied
as particular cases amalgamation (\cite{hhh}, Examples 2.5 and 2.6).\\

Let $\delta: Id(A)\rightarrow~Id(A)$ and $\delta_{1}:Id(f(A)+J)\rightarrow~Id(f(A)+J)$ are expansion functions of $Id(A)$ and $Id(f(A)+J)$ respectively.  We define a function $\delta_{\bowtie^{f}}:Id(A\bowtie^{f}J)\longrightarrow~Id(A\bowtie^{f}J)$ such that $\delta_{\bowtie^{f}}(I\bowtie^{f}J)=\delta(I)\bowtie^{f}J$ for every ideal $I$ of $A$ and $\delta_{\bowtie^{f}}(\overline{K}^{f})=\{(a,f(a)+j)|~a\in A,j\in J, f(a)+j\in\delta_{1}(K)\}$ for every ideal $K$ of $f(A)+J.$ Then, $\delta_{\bowtie^{f}}$ is an expansion function of ideals of $A\bowtie^{f}J.$
\begin{lem}\label{30}
Let $f:A\rightarrow~B$ be a homomorphism of rings and $J$ be an ideal of $B.$  Let $I$ be a proper ideal of $A.$  Then the following statements are equivalent.
\begin{enumerate}
\item[(1)] $I$ is an $(m,n)-$closed $\delta-$primary ideal of $A.$
\item[(2)] $I\bowtie^{f}J$ is an $(m,n)-$closed $\delta_{\bowtie^{f}}-$primary ideal of $A\bowtie^{f}J.$
\end{enumerate}
\end{lem}
\begin{proof}
$(1)\Rightarrow(2).$  Suppose that $I$ is $(m,n)-$closed $\delta-$primary ideal of $A$ and let $(a,f(a)+j)^{m}\in I\bowtie^{f}J$ for some $(a,f(a)+j)\in~A\bowtie^{f}J.$  Then, $a^{m}\in I.$  The fact that $I$ is $(m,n)-$closed $\delta-$primary ideal of $A,$ gives $a^{n}\in \delta(I).$  Which implies $(a,f(a)+j)^{n}\in\delta(I)\bowtie^{f}J=\delta_{\bowtie^{f}}(I\bowtie^{f}J).$  Hence, $I\bowtie^{f}J$ is $(m,n)-$closed $\delta_{\bowtie^{f}}-$primary ideal of $A\bowtie^{f}J.$\\
$(2)\Rightarrow (1).$  Suppose that $I\bowtie^{f}J$ is $(m,n)-$closed $\delta_{\bowtie^{f}}J-$primary ideal of $A\bowtie^{f}J.$ Let $a^{m}\in I$ for some $a\in~A.$  Then, $(a,f(a))^{m}\in I\bowtie^{f}J$ which implies $(a,f(a))^{n}\in\delta_{\bowtie^{f}}(I\bowtie^{f}J)=\delta(I)\bowtie ^{f}J.$  Then, $a^{n}\in\delta(I).$  Hence, $I$ is $(m,n)-$closed $\delta-$primary ideal of $A.$ As desired.
\end{proof}

\begin{thm}\label{31}
Let $f:A\rightarrow B$ be a homomorphism of rings and $J$ be an ideal of $B.$  Then the following statements are equivalent.
\begin{enumerate}
\item[(1)] $I\bowtie^{f}J$ is a weakly $(m,n)-$closed $\delta_{\bowtie^{f}}-$primary ideal of $A\bowtie^{f}J$ that is not $(m,n)-$closed $\delta_{\bowtie^{f}}-$primary.
\item[(2)] $I$ is a weakly $(m,n)-$closed $\delta-$primary ideal of $A$ that is not $(m,n)-$closed $\delta-$primary and for every $\delta$-$(m,n)$-unbreakable-zero element $a$ of $I$ we have $(f(a)+j)^{m}=0$ for every $j\in~J$.
\end{enumerate}
\end{thm}
\begin{proof}
$(1)\Rightarrow(2).$ Assume that $I\bowtie^{f}J$ is weakly $(m,n)-$closed $\delta-$primary ideal of $A\bowtie^{f}J.$  Let $0\neq a^{m}\in I.$ Then, $0\neq(a,f(a))^{m}\in I\bowtie^{f}J$ which implies $(a,f(a))^{n}\in \delta_{\bowtie^{f}}(I\bowtie^{f}J)=\delta(I)\bowtie^{f}J.$  Thus, $a^{n}\in\delta(I)$ and consequently $I$ is weakly $(m,n)-$closed $\delta-$primary ideal of $A.$  On the other hand by the Theorem, we have $I$ is not $(m,n)-$closed $\delta-$primary ideal of $A.$  So there exists an $\delta-(m,n)$-unbreakable-zero element $a$ of $I.$  We will show that $(f(a)+j)^{m}=0$ for every $j$ in $J.$ By the way of contradiction, suppose that $0\neq (f(a)+j)^{m}$ for some $j\in J.$  Then $0\neq (a,f(a)+j)^{m}\in I\bowtie^{f}J.$ As $I\bowtie^{f}J$ is weakly $(m,n)-$closed $\delta_{\bowtie^{f}}-$primary, we get $(a,f(a)+j)^{n}\in \delta_{\bowtie^{f}}(I\bowtie^{f}J),$ which implies $a^{n}\in \delta(I).$  As desired contradiction.\\
$(2)\Rightarrow(1).$ Let $0\neq (a,f(a)+j)^{m}\in I\bowtie^{f}J$ for some $a\in A$ and $j\in J.$  Then, $a^{m}\in I.$  If $0\neq a^{m},$ as $I$ is weakly $(m,n)-$closed $\delta-$primary ideal, we get $a^{n}\in \delta(I).$  Hence, $(a,f(a)+j)^{n}\in\delta_{\bowtie^{f}}(I\bowtie^{f}J).$  Now assume that $a^{m}=0,$ necessarily $a^{n}\in \delta(I),$ if not. We have  $a$ is a $\delta-(m,n)$-unbreakable-zero element of $I.$  So by assumption we have $(f(a)+j)^{m}=0.$  Then, $(a,f(a)+j)^{m}=0,$ a contradiction.  Hence, $I\bowtie^{f}J$ is weakly $(m,n)-$closed $\delta_{\bowtie^{f}}-$primary ideal of $A\bowtie^{f}J.$
\end{proof}
\begin{cor}
Let $f:A\rightarrow~B$ be a homomorphism of rings and $J$ be an ideal of $B.$  If $char(f(A)+J)=m$ and $J^{m}=0.$  Then the following statements are equivalent.
\begin{enumerate}
\item[(1)] $I\bowtie^{f}J$ is a weakly $(m,n)-$closed $\delta-$primary ideal of $A\bowtie^{f}J$ that is not $(m,n)-$closed $\delta-$primary.
\item[(2)] $I$ is a weakly $(m,n)-$closed $\delta$-primary ideal of $A$ that is not $(m,n)-$closed $\delta$-primary.
\end{enumerate}
\end{cor}
\begin{proof}
For every $j\in J,$ we have $$(f(a)+j)^{m}=f(a^{m})=\displaystyle\sum_{k=1}^{m-1}\left(
\begin{array}{cccc}
m\\
k\\
\end{array}
\right)f(a^{m-k}j^{k}+j^{m}.$$
As $J^{m}=0$ we gate $j^{m}=0$ for every $j\in J$ on the other hand we have $\left(
\begin{array}{cccc}
m\\
k\\
\end{array}
\right)=0$ since $ch(f(A)+j)=m$ and $m$ divides $\left(
\begin{array}{cccc}
m\\
k\\
\end{array}
\right).$
\end{proof}
\begin{cor}
Let $A$ be a commutative ring, and $I$ be a proper ideal of $A.$ Let $K$ be a proper ideal of $A.$ Then the following statements are equivalent.
\begin{enumerate}
\item[(1)] $K\bowtie~I$ is a weakly $(m,n)-$closed $\delta_{\bowtie}-$primary ideal of $A\bowtie~I$ which is not $(m,n)-$closed $\delta_{\bowtie}-$primary.
\item[(2)] $K$ is weakly $(m,n)-$closed $\delta-$primary ideal of $A$ which is not $(m,n)-$closed $\delta-$primary and $(a+i)^{m}=0$ for every an $(m,n)$-unbreakable element $a$ of $K$ and every element $i\in~I$.
\end{enumerate}
\end{cor}
\begin{proof}
Take $A=B$ and $f=id_{A}$ in Theorem \ref{31}, where $id_{A}$ is the identity map $id_{A}:A\rightarrow~A$.
\end{proof}
Now, we give another proof to Theorem \ref{34} by using amalgamated algebra.  Recall Theorem \ref{34}:
\begin{cor}\label{32}
Let $A$ be a commutative ring, $M$ be an $A$-module.  Let $I$ be a proper ideal of $A.$  Then the following statements are equivalent.
\begin{enumerate}
\item[(1)] $I(+)M$ is a weakly $(m,n)-$closed $\delta_{(+)}-$primary ideal of $A(+)M$ that is not $(m,n)-$closed $\delta_{(+)}-$primary.
\item[(2)] $I$ is a weakly $(m,n)-$closed $\delta-$primary ideal of $A$ that is not $(m,n)-$closed $\delta-$primary and $m(a^{m-1})M=0$ for every $(m,n)$-unbreakable element $a$ of $I.$
\end{enumerate}
\end{cor}
\begin{proof}
Let $f:A\rightarrow B$ be the canonical homomorphism defined by, for every $a\in A,~f(a)=(a,0)$ and $J:=0\propto~M.$  It not difficult to check that $A\propto~E$ is naturally isomorphic to $A\bowtie^{f}J,$ and the ideal $I\bowtie^{f}J$ is canonically isomorphic to $I\propto~M.$  By Theorem \ref{31}, we have $I(+)M$ is weakly $(m,n)-$closed $\delta_{(+)}-$primary ideal of $A(+)M$ which is not $(m,n)-$closed $\delta_{(+)}-$primary if and only if $I$ is weakly $(m,n)-$closed $\delta-$primary ideal of $A$ which is not $(m,n)-$closed $\delta-$primary and for ever $(m,n)$-unbreakable element $a$ of $I$, we have $(f(a)+j)^{m}=0$ for every $j\in J=(0)(+) M.$  Now, if $x\in M,$ then $((a,0)+(0,x))^{m}=(a^{m},ma^{m-1}x)=0,$ thus $m(a^{m-1}x)=0.$  Hence, $m(a^{m-1}M)=0.$
\end{proof}
\begin{rem}
If $I$ is a weakly $(m,n)-$closed $\delta-$primary ideal of $A,$ then $I\bowtie^{f}B$ need not to be a weakly $(m,n)-$closed $\delta_{\bowtie^{f}}-$primary ideal of $A\bowtie^{f}J.$  Let $A=Z_{8}.$  Then, $I=\{0\}$ is clearly weakly $(3,1)-$closed $\delta_{\sqrt{I}}-$primary ideal that not $(3,1)-$closed $\delta_{\sqrt{I}}-$primary, since $2^{3}\in\{0\}$ but $2^{2}\notin\{0\}.$  Let $M$ be an $A$-module and set $J:=0(+)M.$  Let $f:A\hookrightarrow~A(+)M$ the canonical homomorphism defined by $f(a)=(a,0)$ for every $a\in A.$ It is clear to see that $0\bowtie^{f}J$ is isomorphic to $0(+)M.$  Since $3(2^{2})(0(+)M)\neq 0$ by the Corollary \ref{32}, we conclude that $0\bowtie^{f}J$ is not a weakly $(3,1)-$closed $\delta_{\bowtie^{f}}-$primary ideal of $A\bowtie^{f}J.$
\end{rem}

\begin{lem}\label{33}
$\overline{K}^{f}$ is $(m,n)-$closed $\delta_{\bowtie^{f}}-$primary ideal of $A\bowtie^{f}J$ if and only if $K$ is $(m,n)-$closed $\delta_{1}-$primary ideal of $f(A)+J.$
\end{lem}
\begin{proof}
Suppose that $\overline{K}^{f}$ is an $(m,n)-$closed $\delta_{\bowtie^{f}}-$primary ideal of $A\bowtie^{f}J.$  We claim that $K$ is an $(m,n)-$closed $\delta_{1}-$primary ideal of $f(A)+J.$  Indeed, let $(f(a)+j)^{m}\in K$ with $f(a)+j\in f(A)+J.$  Thus, $(a,f(a)+j)^{m}\in\overline{K}^{f}.$ Since $\overline{K}^{f}$ is an $(m,n)-$closed $\delta_{\bowtie^{f}}-$primary ideal, $(a,f(a)+j)^{n}\in\delta_{\bowtie^{f}}(\overline{K}^{f})=\{(a,f(a)+j)|~a\in~A,j\in J,~f(a)+j\in\delta_{1}(K)\}.$  Therefore, $(f(a)+j)^{n}\in\delta_{1}(K).$  Hence, $K$ is an $(m,n)-$closed $\delta_{1}-$primary ideal of $f(A)+J.$  Conversely, assume that $K$ is an $(m,n)-$closed $\delta_{1}-$primary ideal of $f(A)+J.$ Let $(a,f(a)+j)^{m}\in\overline{K}^{f}$ with $(a,f(a)+j)\in A\bowtie^{f}J.$  Obviously, $f(a)+j\in f(A)+j$ and $(f(a)+j)^{m}\in K$ which is an $(m,n)-$closed $\delta_{1}-$primary ideal.  So $(f(a)+j)^{n}\in\delta_{1}(K).$  Which implies $(a,f(a)+j)^{n}\in\delta_{\bowtie^{f}}(\overline{K}^{f}).$  Hence, $\overline{K}^{f}$ is an $(m,n)-$closed $\delta_{\bowtie^{f}}-$primary ideal of $A\bowtie^{f}J,$ as desired.
\end{proof}
\begin{thm}
The following statements are equivalent.
\begin{enumerate}
\item[(1)] $\overline{K}^{f}$ is a weakly $(m,n)-$closed $\delta_{\bowtie^{f}}-$primary ideal of $A\bowtie^{f}J$ which is not $(m,n)-$closed $\delta_{\bowtie^{f}}-$primary.
\item[(2)] $K$ is a weakly $(m,n)-$closed $\delta_{1}-$primary ideal of $f(A)+J$ which is not $(m,n)-$closed $\delta_{1}-$primary and for every $(m,n)$-unbreakable-zero $f(a)+j$ of $K$ we have $a^{m}=0.$
\end{enumerate}
\end{thm}
\begin{proof}
$(1)\Rightarrow (2).$ Suppose $\overline{K}^{f}$ is a weakly $(m,n)-$closed $\delta_{\bowtie^{f}}-$primary ideal of $A\bowtie^{f}J.$  We claim that $K$ is a weakly $(m,n)-$closed $\delta_{1}-$primary ideal of $f(A)+J.$  Indeed, let $0\neq (f(a)+j)^{m}\in K$ with $f(A)+j\in f(A)+J.$  Then, $0\neq(a,f(a)+j)^{m}\in\overline{K}^{f}.$  Since $\overline{K}^{f}$ is a weakly $(m,n)-$closed $\delta_{\bowtie^{f}}-$primary ideal of $A\bowtie^{f}J,$ we have $(a,f(a)+j)^{n}\in\overline{K}^{f}.$  Therefore, $f((a)+j)^{n}\in K.$  Hence, $K$ is a weakly $(m,n)-$closed $\delta_{1}-$primary ideal of $f(A)+J.$  By Lemma \ref{33}, $K$ is not $(m,n)-$closed $\delta_{1}-$primary ideal of $f(A)+J.$  Now, let $f(a)+j\in f(A)+J$ be an $(m,n)$-unbreakable-zero element of $K.$  We claim that $a^{m}=0.$  Indeed if $0\neq a^{m},$ we get $(a,f(a)+j)^{m}\in\overline{K}^{f},$ then $(a,f(a)+j)^{n}\in\overline{K}^{f}$, which implies $(f(a)+j)^{n}\in~K.$  This is a contradiction.  Hence, $a^{m}=0.$\\
$(2)\Rightarrow(1).$ Suppose that $K$ is a weakly $(m,n)-$closed $\delta_{1}-$primary ideal of $f(A)+J.$  Let $0\neq (a,f(a)+j)^{m}\in\overline{K}^{f}$ for some $(a,f(a)+j)\in~A\bowtie^{f}J.$  Then, obviously, $f(a)+j\in f(A)+J$ and $(f(a)+j)^{m}\in K.$  If $0\neq (f(a)+j)^{m},$ as $K$ is weakly $(m,n)-$closed $\delta_{1}-$primary ideal of $f(A)+J,$ we get $(f(a)+j)^{n}\in K,$ which implies $(a,f(a)+j)^{n}\in \overline{K}^{f}.$  Now, if $(f(a)+j)^{m}=0.$  We claim that $(f(a)+j)^{n}\in K,$ if not, we get $f(a)+j$ is an $(m,n)$-unbreakable-zero.  Then by assumption, we have $a^{m}=0 $ and thus $(a,f(a)+j)^{m}=0.$  This is a contradiction.  Thus, $(f(a)+j)^{n}\in K.$ Therefore, $(a,f(a)+j)^{n}\in\overline{K}^{f}.$  Hence, $\overline{K}^{f}$ is a weakly $(m,n)-$closed $\delta_{\bowtie^{f}}-$primary ideal of $A\bowtie^{f}J.$  By Lemma \ref{33}, we have $\overline{K}^{f}$ is not $(m,n)$closed $\delta_{\bowtie^{f}}-$primary ideal of $A\bowtie^{f}J.$
\end{proof}

\section{Conclusion}
This article included the structure of weakly $(m,n)-$closed $\delta-$primary ideals of a commutative ring $R$.  We have discussed and proved important results in this class of ideals.  We proved that if $R$ is a commutative ring, $\delta$ an expansion function of $Id(R)$, $m$ and $n$ positive integers, and $I$ a weakly $(m,n)-$ closed $\delta-$primary ideal of $R$ that is not $(m,n)-$ closed $\delta-$primary, then $I\subseteq~Nil(R)$.  Moreover, if $char(R)=m$ is prime, then $x^m=0$ for every $x\in~I$.  Also, we have studied the weakly $(m,n)-$ closed $\delta-$primary ideal $I$ of the trivial ring extension of $R$ by an $R-$module $M$.  We transferred the notion of weakly $(m,n)-$closed $\delta-$primary ideals to the amalgamated algebras along an ideal.\\
\textbf{Acknowledgement}\\
The authors are grateful to the reviewers for their helpful comments and suggestions aimed at improving this paper.

\end{document}